\documentclass[10pt, oneside, reqno]{amsart}
\newtheorem{thm}{Theorem}[section]

\newtheorem{main-thm}{Main Theorem}
\newtheorem{prop}[thm]{Proposition}
\newtheorem{cor}[thm]{Corollary}

\theoremstyle{definition}
\newtheorem{defn}[thm]{Definition}

\numberwithin{equation}{section}
\usepackage{verbatim}
\usepackage{mathrsfs}
\usepackage{graphicx}
\usepackage{pinlabel}
\usepackage{lineno}

\begin{document}
\title{Counting Generic Quadrisecants of Polygonal Knots}
\author{Aldo-Hilario Cruz-Cota, Teresita Ramirez-Rosas}
\address{Department of Mathematics and Statistics, University of Nebraska at Kearney, Kearney, NE 68845, USA}
\email{cruzcotaah@unk.edu, tere.ramirez.rosas@gmail.com}

\begin{abstract} Let $K$ be a polygonal knot in general position with vertex set $V$. A \emph{generic quadrisecant} of $K$ is a line that is disjoint from the set $V$ and intersects $K$ in exactly four distinct points. We give an upper bound for the number of generic quadrisecants of a polygonal knot $K$ in general position. This upper bound is in terms of the number of edges of $K$.
\end{abstract}


\maketitle

\section{Introduction}

In this article, we study polygonal knots in three dimensional space that are in general position. Given such a knot $K$, we define a \emph{quadrisecant} of $K$ as an unoriented line that intersects $K$ in exactly four distinct points. We require that these points are not vertices of the knot, in which case we say that the quadrisecant is \emph{generic}.

Using geometric and combinatorial arguments, we give an upper bound for the number of generic quadrisecants of a polygonal knot $K$ in general position. This bound is in terms of the number $n \geq 3$ of edges of $K$. More precisely, we prove the following.

\begin{main-thm} \label{main-thm-1}
Let $K$ be a polygonal knot in general position, with exactly $n$ edges. Then $K$ has at most $U_n=\dfrac{n}{12}(n-3)(n-4)(n-5)$ generic quadrisecants.
\end{main-thm}

Applying Main Theorem \ref{main-thm-1} to polygonal knots with few edges, we obtain the following.

\begin{enumerate}
 \item If $n \leq 5$, then $K$ has no generic quadrisecant.
 \item If $n=6$, then $K$ has at most three generic quadrisecants.
 \item If $n=7$, then $K$ has at most $14$ generic quadrisecants.
\end{enumerate}

Using a result of G. Jin and S. Park (\cite{Taek-1}), we can prove that the above bound is sharp for $n=6$. In other words, a hexagonal trefoil knot has exactly three quadrisecants, all of which are generic.

Quadrisecants of polygonal knots in $\mathbb{R}^3$ have been studied by many people, such as E. Pannwitz, H. Morton, D. Mond, G. Kuperberg and E. Denne. The study of quadrisecants started in $1933$ with E. Pannwitz's doctoral dissertation (\cite{Pann33}). There, she found a lower bound for the number of quadrisecants of non-trivial generic polygonal knots. This bound is in terms of the minimal number of boundary singularities for a disk bounded by $K$. Later, H. Morton and D. Mond (\cite{Morton}) proved that every non-trivial generic knot has a quadrisecant, and G. Kuperberg extended their result to non-trivial tame knots and links (\cite{Kuperberg}). More recently, E. Denne (\cite{dennethesis}) proved that essential alternating quadrisecants exist for all non-trivial tame knots.
\subsection*{Notation} 
Unless otherwise stated, all polygonal knots studied in this article are embedded in the three-dimensional Euclidean space $\mathbb{R}^3$. Such a knot will be denoted by $K$. The cardinality of a set $A$ is denoted by $|A|$. Given a set $A$, with $|A|=n$, ${n \choose k}$ denotes the number of subsets of $A$ of cardinality $k$. The symbol $\sqcup$ denotes the disjoint union of sets.

\section{Preliminaries}

It is well-known that a triple of pairwise skew lines $E_1,E_2,E_3$ determines a unique quadric. This quadric is a doubly-ruled surface $S$ that is either a hyperbolic paraboloid, if the three lines are parallel to one plane, or a hyperboloid of one sheet, otherwise (see for example \cite{Geom-Imag-Hilbert}). The lines $E_1,E_2,E_3$ belong to one of the rulings of $S$, and every line intersecting all those three lines belongs to the other ruling of $S$. Further, every point in $S$ lies on a unique line from each ruling (see \cite{dennethesis}, \cite{Otal-prop-elem} and \cite{Pottmann-comp-lin-geom}).

We now define the type of polygonal knots that we will consider in this article.

\begin{defn} \label{def-knot-gen-pos} We say that the polygonal knot $K$ in $\mathbb{R}^3$ is in \emph{general position} if the following conditions are satisfied:
\begin{itemize}
 \item No four vertices of $K$ are coplanar.
 \item Given three edges $e_1,e_2,e_3$ of $K$ that are pairwise skew, no other edge of $K$ is contained in the quadric generated by $e_1,e_2,e_3$. 
\end{itemize}
\end{defn}

The quadrisecants of knots that we will study are defined as follows.

\begin{defn} Let $K$ be a polygonal knot in general position with vertex set $V$. A \emph{generic quadrisecant} of $K$ is an unoriented line that is disjoint from the set $V$ and intersects $K$ in exactly four distinct points.
\end{defn}

In this paper we are interested in giving an upper bound for the number of generic quadrisecants of a polygonal knot $K$ in general position. This upper bound is in terms of the number of edges of $K$. We start by estimating the number of generic quadrisecants that intersect a given collection of four edges of $K$ that are pairwise skew.

\begin{prop} \label{No-3-QS-skew-edges} Let $K$ be a knot in general position. Let $\mathcal{E}_4$ be a collection of four distinct edges of $K$ that are pairwise skew. Then there are at most two generic quadrisecants of $K$ that intersect all edges in  $\mathcal{E}_4$.
\end{prop}

\begin{proof}
Let $e_1,e_2,e_3,e_4$ be the four edges in the collection $\mathcal{E}_4$. Each edge $e_i$ generates a line $E_i$ ($i=1,2,3,4)$. Let $S$ be the doubly-ruled quadric generated by $E_1,E_2,E_3$. Since $K$ is in general position, the edge $e_4$ is not contained in $S$. Therefore, $e_4$ intersects the quadric $S$ in at most two points.  

Let $\mathcal{Q}_{\mathcal{E}_4}$ be the set of all generic quadrisecants of $K$ that intersect all edges in  $\mathcal{E}_4$. For $l \in \mathcal{Q}_{\mathcal{E}_4}$, we define the point $p_l$ as the point of intersection between the edge $e_4$ and the line $l$. Since $l$ intersects all lines $E_1,E_2,E_3$, then it belongs to a ruling $\mathcal{R}$ of $S$. Also, $p_l \in e_4 \cap S$, and so the cardinality of the set $\{p_l \colon l \in \mathcal{Q}_{\mathcal{E}_4} \}$ is at most two. To complete the proof, we show that the function $l \mapsto p_l$ is one-to-one. Suppose that $p_l=p_{l'}$, where $l \in \mathcal{Q}_{\mathcal{E}_4}$ and $l' \in \mathcal{Q}_{\mathcal{E}_4}$. Then the point $p_l=p_{l'}$ lies in two lines, $l$ and $l'$, that belong to the ruling $\mathcal{R}$ of $S$. Since every point in $S$ lies on a unique line from $\mathcal{R}$, then $l=l'$.
\end{proof}

Our next result complements Proposition \ref{No-3-QS-skew-edges}.

\begin{prop} \label{No-2-QS-coplanar-edges} Let $K$ be a knot in general position. Let $\mathcal{E}_4$ be a collection of four distinct edges of $K$, two of which are coplanar. Then there is at most one generic quadrisecant of $K$ that intersects all edges in  $\mathcal{E}_4$.
\end{prop}

\begin{proof} Let $e_1,e_2,e_3,e_4$ be the four edges in the collection $\mathcal{E}_4$, and suppose that $e_1$ and $e_2$ lie in a plane $P$.  By general position, $e_1$ and $e_2$ are adjacent edges. Arguing toward a contradiction, suppose that $l_1$ and $l_2$ are two distinct generic quadrisecants of $K$ that intersect all edges in  $\mathcal{E}_4$. 

Since $e_1$ and $e_2$ lie in $P$, then the same is true for $l_1$ and $l_2$. Since both $l_1$ and $l_2$ intersect the edge $e_i$, then so does $P$ ($i=3,4$). By general position, the edge $e_i$ intersects $P$ in a single point $p_i$, which is a point of intersection between the lines $l_1$ and $l_2$ ($i=3,4$). Thus, $p_3=p_4$, and so the edge $e_3$ intersects the edge $e_4$. This means that the point $p_3=p_4$ is a vertex of both $e_3$ and $e_4$, and this vertex is different from those of edges $e_1$ and $e_2$ (because $K$ is a knot). This contradicts general position.

\end{proof}

\section{Quadrisecants Intersecting Consecutive Edges of the Knot}

To prove some of the results in the next section, we will need to analyze collections of edges of a polygonal knot that have the property defined below.

\begin{defn} Let $\mathcal{E}'$ be a collection of distinct edges of a polygonal knot $K$. We will say that the edges in $\mathcal{E}'$ are \emph{consecutive} if their union  (with the subspace topology induced from $K$) is connected.
\end{defn}

Since two consecutive edges of a polygonal knot are always coplanar, then Proposition \ref{No-2-QS-coplanar-edges} implies the following. 

\begin{prop} \label{No-2-QS-int-2-cons-edges}  Let $K$ be a knot in general position. Let $\mathcal{E}_4$ be a collection of four distinct edges of $K$ that contains a pair of consecutive edges. Then there is at most one generic quadrisecant of $K$ that intersects all edges in  $\mathcal{E}_4$.
\end{prop}

We now investigate the existence of generic quadrisecants intersecting two or three consecutive edges of a polygonal knot.

\begin{prop} \label{No-QS-int-3-cons-edges}
 There are no generic quadrisecants of $K$ intersecting three distinct consecutive edges of $K$.
\end{prop}

\begin{proof} Let $n$ be the number of edges of $K$. If $n=3$, then the result is clear. Suppose that $n>3$ and that  $l$ is a generic quadrisecant that intersects three distinct consecutive edges of $K$. Then the plane $P$ that contains $l$ and one of the three consecutive edges also contains the other two edges. Since $n>3$, then the endpoints of the three consecutive edges are four distinct vertices of $K$, and these vertices lie in the plane $P$. This contradicts that $K$ is in general position.

\end{proof}

Proposition \ref{No-QS-int-3-cons-edges} has the following immediate corollary.

\begin{cor} \label{No-QS-int-4-cons-edges}
  There are no generic quadrisecants of $K$ intersecting four distinct consecutive edges of $K$.
\end{cor}

The following proposition complements Proposition \ref{No-2-QS-int-2-cons-edges}.

\begin{prop} \label{No-3-QS-int-no-cons-edges} Let $\mathcal{E}_4$ be a collection of four distinct edges of $K$ that contains no pair of consecutive edges. Then there are at most two generic quadrisecants of $K$ that intersect all edges in  $\mathcal{E}_4$.
\end{prop}

\begin{proof}
 If all edges in $\mathcal{E}_4$ are pairwise skew, then Proposition \ref{No-3-QS-skew-edges} implies that there are at most two generic quadrisecants of $K$ intersecting all edges in  $\mathcal{E}_4$. If the collection $\mathcal{E}_4$ contains a pair of coplanar edges, then Proposition \ref{No-2-QS-coplanar-edges} implies that there is at most one generic quadrisecant of $K$ intersecting all edges in  $\mathcal{E}_4$.
\end{proof}

\section{Combinatorial Results} \label{sec-comb-res}

For a collection $\mathcal{E}_4$ of four distinct edges of the knot $K$, the following theorem gives an upper bound for the number of generic quadrisecants of $K$ that  intersect all edges in $\mathcal{E}_4$.

\begin{thm} \label{thm-all-cases-c} Let $K$ be a polygonal knot in general position.  Given a collection $\mathcal{E}_4$ of four distinct edges of $K$, consider the union $X_{\mathcal{E}_4}$ of the edges in $\mathcal{E}_4$ (with the subspace topology induced from $K$). Let $c$ be the number of connected components of the space $X_{\mathcal{E}_4}$. 
 \begin{itemize}
  \item If $c=1$, then there are no generic quadrisecants intersecting all edges in  $\mathcal{E}_4$.
  \item If $c=2$, and one of the connected components of $X_{\mathcal{E}_4}$ consists of a single edge of $K$, then there are no generic quadrisecants intersecting all edges in  $\mathcal{E}_4$.
  \item If $c=2$, and each of the connected components of $X_{\mathcal{E}_4}$ is the union of exactly two consecutive edges of $K$, then there is at most one generic quadrisecant intersecting all edges in  $\mathcal{E}_4$.
  \item If $c=3$, then there is at most one generic quadrisecant intersecting all edges in  $\mathcal{E}_4$.
  \item If $c=4$, then there are at most two generic quadrisecants intersecting all edges in  $\mathcal{E}_4$.
 \end{itemize}
\end{thm}

\begin{proof} We divide the proof into four cases.

\begin{description}
 \item [Case 1] $c=1$. In this case Corollary \ref{No-QS-int-4-cons-edges} implies the result.
  \item [Case 2] $c=2$. If one of the connected components of $X_{\mathcal{E}_4}$ consists of a single edge, then the result follows from Proposition \ref{No-QS-int-3-cons-edges}. Otherwise, the result follows from Proposition \ref{No-2-QS-int-2-cons-edges}.
  \item [Case 3] $c=3$. Since the collection $\mathcal{E}_4$ contains a pair of consecutive edges, then Proposition \ref{No-2-QS-int-2-cons-edges} implies the result.
    \item [Case 4] $c=4$. Since $\mathcal{E}_4$ contains no pair of consecutive edges, then the result follows from Proposition \ref{No-3-QS-int-no-cons-edges}.
\end{description}
\end{proof}

To obtain an upper bound for the number of generic quadrisecants of a knot, we need to consider the number of collections of four distinct edges of the knot for each of the cases stated in Theorem \ref{thm-all-cases-c}. These numbers are defined as follows.

\begin{defn} \label{defn-Sc} Let $K$ be a polygonal knot in general position with exactly $n$ edges. For a collection $\mathcal{E}_4$ of four distinct edges of $K$, consider the union $X_{\mathcal{E}_4}$ of the edges in $\mathcal{E}_4$ (with the subspace topology induced from $K$).
\begin{itemize}
 \item For $c=1,2,3,4$, let $S^{(n)}_{c}(K)$ be the number of collections $\mathcal{E}_4$ of four distinct edges of $K$ such that $X_{\mathcal{E}_4}$ has exactly $c$ connected components.
 \item For $c=2$ we also define the following.
 
 \begin{itemize}
 \item Let $S^{(n)}_{2,1}(K)$ be the number of collections $\mathcal{E}_4$ of four distinct edges of $K$ such that $X_{\mathcal{E}_4}$ has exactly two connected components, and one of these components consists of a single edge.
  \item Let $S^{(n)}_{2,2}(K)$ be the number of collections $\mathcal{E}_4$ of four distinct edges of $K$ such that $X_{\mathcal{E}_4}$ has exactly two connected components, and each of these components is the union of exactly two consecutive edges.
 \end{itemize}
\end{itemize}
\end{defn}

By definition,
 \begin{equation} \label{eqn-sum-S_c=2}
 S^{(n)}_{2}(K)=S^{(n)}_{2,1}(K)+S^{(n)}_{2,2}(K);
 \end{equation}
 
 \begin{equation} \label{eqn-sum-S_c's}
 S^{(n)}_{1}(K)+S^{(n)}_{2,1}(K)+S^{(n)}_{2,2}(K)+S^{(n)}_{3}(K)+S^{(n)}_{4}(K)={n \choose 4}.
 \end{equation}

Combining Theorem \ref{thm-all-cases-c} with Definition \ref{defn-Sc}, we obtain an upper bound for the number of generic quadrisecants of a polygonal knot in general position.

\begin{cor} \label{cor-upper-bound-QS} Let $K$ be a polygonal knot in general position with exactly $n$ edges. Then the  number $U_n= S^{(n)}_{2,2}(K)+S^{(n)}_{3}(K)+2S^{(n)}_{4}(K)$ is an upper bound for the number of generic quadrisecants of $K$.
\end{cor}

In our next result we find explicit formulas for the numbers  $S^{(n)}_{c}(K)$'s.

\begin{thm} \label{thm-comp-S_c's} Let $K$ be a polygonal knot in general position with exactly $n$ edges. Then

\begin{equation} \label{eqn-S_c-c=1}
S^{(n)}_{1}(K) = \begin{cases} 0 &\mbox{if } n=3\\
1 & \mbox{if } n=4\\
n & \mbox{if } n \geq 5;\\ \end{cases}
\end{equation}

\begin{equation} \label{eqn-S_c-c=2,1}
S^{(n)}_{2,1}(K) = \begin{cases} 0 &\mbox{if } n \leq 5\\
n(n-5) & \mbox{if } n \geq 6;\\ \end{cases}
\end{equation}

\begin{equation} \label{eqn-S_c-c=2,2}
S^{(n)}_{2,2}(K) = \begin{cases} 0 &\mbox{if } n \leq 5\\
\frac{n(n-5)}{2} & \mbox{if } n \geq 6;\\ \end{cases}
\end{equation}

\begin{equation} \label{eqn-S_c-c=3}
S^{(n)}_{3}(K) = \begin{cases} 0 &\mbox{if } n \leq 6\\
\frac{n(n-5)(n-6)}{2} & \mbox{if } n \geq 7;\\ \end{cases}
\end{equation}

\begin{equation} \label{eqn-S_c-c=4}
S^{(n)}_{4}(K) = \begin{cases} 0 &\mbox{if } n \leq 7\\
{n \choose 4}-\frac{n(n-5)(n-6)}{2}-\frac{n(n-5)}{2}-n(n-5)-n & \mbox{if } n \geq 8.\\ \end{cases}
\end{equation}

\end{thm}

\begin{proof} Fix an orientation of $K$ and an edge $e_1$ of $K$. Suppose that $e_1,e_2,\cdots,e_n$ (in that order) are all the distinct edges of $K$ that we encounter when we follow the orientation of $K$, starting and ending at the initial point of $e_1$. For the rest of the proof, the subindices of the edges $e_j$'s are understood modulo $n$. 

\begin{description}
 \item [Proof of equation \eqref{eqn-S_c-c=1}] Clearly, $S^{(n)}_{1}(K)=0$ for $n=3$ and $S^{(n)}_{1}(K)=1$ for $n=4$. Suppose that $n \geq 5$.  Let $\mathcal{E}_4$ be a collection of four distinct edges of $K$ such that $X_{\mathcal{E}_4}$ is connected. The collection $\mathcal{E}_4$ is completely determined by the only integer $i \in \{1,2,\cdots,n\}$ such that $\mathcal{E}_4=\{e_i,e_{i+1},e_{i+2},e_{i+3}\}$. Since this number $i$ can be chosen in $n$ different ways, then $S^{(n)}_{1}(K)=n$.
 
 \item [Proof of equation \eqref{eqn-S_c-c=2,1}]
 If $n \leq 5$, then clearly $S^{(n)}_{2,1}(K)=0$ and $S^{(n)}_{2,2}(K)=0$. For the proof of equations \eqref{eqn-S_c-c=2,1} and \eqref{eqn-S_c-c=2,2}, we will assume that $n \geq 6$. 

 Let $\mathcal{E}_4$ be a collection of four distinct edges of $K$ such that $X_{\mathcal{E}_4}$ has exactly two connected components, $X_1$ and $X_2$, with $X_1$ consisting of a single edge of $K$. Let $\mathcal{E}_3$ be the collection of the three consecutive edges in $X_2$. There are $n$ different ways to choose  the collection $\mathcal{E}_3$. Once we have chosen the three edges $e_i,e_{i+1},e_{i+2}$ in $X_2$, the edge in $X_1$  has to be different from the edges $e_{i-1},e_i,e_{i+1},e_{i+2},e_{i+3}$. Thus, given the edges in $X_2$, the edge in $X_1$ can be chosen in $n-5$ different ways.  Hence, the number $S^{(n)}_{2,1}(K)$ is equal to $n(n-5)$.

\item [Proof of equation \eqref{eqn-S_c-c=2,2}] We may assume that $n \geq 6$. Let $\mathcal{E}_4$ be a collection of four distinct edges of $K$ such that $X_{\mathcal{E}_4}$ has exactly two connected components, $X_1$ and $X_2$, with each $X_i$ being the union of exactly two consecutive edges of $K$. There are $n$ different ways to choose the collection of edges in $X_1$. Once we have chosen the two edges in $X_1$, the edges in $X_2$ can be chosen in $n-5$ different ways. However we are double-counting, as interchanging the collections $X_1$ and $X_2$ produces the same collection $X_{\mathcal{E}_4}$. Therefore, $S^{(n)}_{2,2}(K)=\frac{n(n-5)}{2}$.
\item [Proof of equation \eqref{eqn-S_c-c=3}] We may assume that $n \geq 7$. Let $\mathcal{E}_4$ be a collection of four distinct edges of $K$ such that $X_{\mathcal{E}_4}$ has exactly three connected components, $X_1$, $X_2$ and $X_3$, with $X_1$ being the union of exactly two edges of $K$. There are $n$ different ways to choose the collection of edges in $X_1$. Once we have chosen the two edges in $X_1$, the two edges in $X_2 \sqcup X_3$ can be chosen in  ${n-4 \choose 2}-k$ different ways, where $k$ is the number of different ways to choose a collection of two consecutive edges out of $n-4$ edges. Since $k=n-5$, then the collection of edges in $X_2 \sqcup X_3$ can be chosen in  ${n-4 \choose 2}-(n-5)=\frac{(n-5)(n-6)}{2}$ different ways. Hence, the number $S^{(n)}_{3}(K)$ is equal to $\frac{n(n-5)(n-6)}{2}$.
 
\item [Proof of equation \eqref{eqn-S_c-c=4}] We may assume that $n \geq 8$. By equation \eqref{eqn-sum-S_c's}, \[S^{(n)}_{4}(K)={n \choose 4}-S^{(n)}_{1}(K)-S^{(n)}_{2,1}(K)-S^{(n)}_{2,2}(K)-S^{(n)}_{3}(K).\] Thus, equation \eqref{eqn-S_c-c=4} follows from equations \eqref{eqn-S_c-c=1} to \eqref{eqn-S_c-c=3}.
\end{description}
 \end{proof}

\section{The Main Result}

Combining Corollary \ref{cor-upper-bound-QS} with Theorem \ref{thm-comp-S_c's}, we obtain an \emph{explicit} upper bound for the number of generic quadrisecants of a polygonal knot in general position.

\begin{cor} \label{cor-explicit-upper-bound-QS} Let $K$ be a polygonal knot in general position with exactly $n$ edges.
\begin{enumerate}
 \item If $n \leq 5$, then $K$ has no generic quadrisecant.
 \item If $n=6$, then $K$ has at most three generic quadrisecants.
 \item If $n=7$, then $K$ has at most $14$ generic quadrisecants.
 \item If $n \geq 8$, then $K$ has at most $\dfrac{n}{12}(n-3)(n-4)(n-5)$ generic quadrisecants.
\end{enumerate}
\end{cor}

\begin{proof} By Corollary \ref{cor-upper-bound-QS}, the knot $K$ has at most $U_n= S^{(n)}_{2,2}(K)+S^{(n)}_{3}(K)+2S^{(n)}_{4}(K)$ generic quadrisecants. 
\begin{enumerate}
\item Suppose that $n \leq 5$. Then $S^{(n)}_{2,2}(K)=0=S^{(n)}_{3}(K)=S^{(n)}_{4}(K)$, and so $U_n=0$.
\item Suppose that $n=6$. Then $S^{(6)}_{2,2}(K)=3$, $S^{(6)}_{3}(K)=0$ and $S^{(6)}_{4}(K)=0$, so $U_n=3$.
\item Suppose that $n=7$. Then $S^{(7)}_{2,2}(K)=7$, $S^{(7)}_{3}(K)=7$ and $S^{(7)}_{4}(K)=0$, so $U_n=14$.
\item Suppose that $n \geq 8$. By equation \eqref{eqn-sum-S_c's}, \[S^{(n)}_{4}(K)={n \choose 4}-S^{(n)}_{1}(K)-S^{(n)}_{2,1}(K)-S^{(n)}_{2,2}(K)-S^{(n)}_{3}(K).\] Thus, 
 \begin{equation} \label{eqn-pf-cor-explicit-upper-bound-QS-1}
  U_n=2{n \choose 4}-2S^{(n)}_{1}(K)-2S^{(n)}_{2,1}(K)-S^{(n)}_{2,2}(K)-S^{(n)}_{3}(K).
 \end{equation}

By Theorem \ref{thm-comp-S_c's}, equation \eqref{eqn-pf-cor-explicit-upper-bound-QS-1} becomes:

\begin{equation} \label{eqn-pf-cor-explicit-upper-bound-QS-2}
U_n=\frac{1}{12}n(n-1)(n-2)(n-3)-2n-2n(n-5)-\frac{n(n-5)}{2}-\frac{n(n-5)(n-6)}{2}.
\end{equation}

Equation \eqref{eqn-pf-cor-explicit-upper-bound-QS-2} can be written as $\frac{n}{12}(n-3)(n-4)(n-5)$.
\end{enumerate}

\end{proof}

Notice that the expression $\dfrac{n}{12}(n-3)(n-4)(n-5)$ from Corollary \ref{cor-explicit-upper-bound-QS} is equal to zero for $n=3,4,5$; it is equal to three for $n=6$, and it is equal to $14$ for $n=7$. This means that Corollary \ref{cor-explicit-upper-bound-QS} can be reformulates as follows.

\begin{thm} \label{thm-upper-bound-QS-single-form}
Let $K$ be a polygonal knot in general position with exactly $n$ edges. Then $K$ has at most $U_n=\dfrac{n}{12}(n-3)(n-4)(n-5)$ generic quadrisecants.
\end{thm}

\subsection{Examples}
\begin{itemize}
 \item Let $K$ be a trivial knot with exactly $n \leq 5$ edges. Then $U_3=0$, so $K$ has no generic quadrisecants.
 \item Let $K$ be a hexagonal trefoil knot. Then $U_6=3$, so $K$ has at most three generic quadrisecants. In this case, the upper bound given by $U_6$ is sharp, as $K$ has exactly three generic quadrisecants (see \cite{Taek-1}). Figure \ref{allquadrisecants} shows a hexagonal trefoil knot with its three generic quadrisecants.
 \item Let $K$ be a heptagonal figure eight knot. We have found that $K$ has at least 6 generic quadrisecants. On the other hand, the upper bound given by $U_7$ may not be sharp in this case, as $U_7=14$.

\begin{figure}[ht!]
\centering
\includegraphics[scale=0.7]{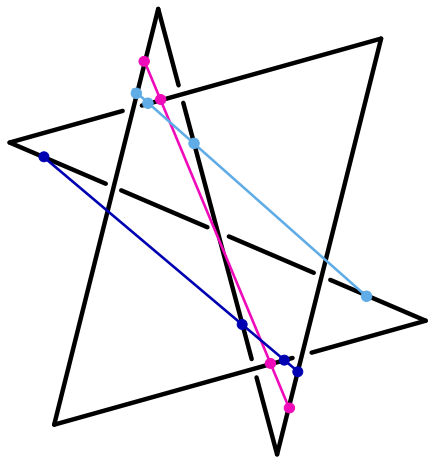}
\caption{The three quadrisecants of a hexagonal trefoil knot.}
\label{allquadrisecants}
\end{figure}
\end{itemize}

\bibliographystyle{abbrv}
\bibliography{biblio}

\end{document}